\documentclass{amsart}
\usepackage{amssymb}

\newcommand{\al}{\alpha}
\newcommand{\be}{\beta}
\newcommand{\Bf}{\mathbf{B}}
\newcommand{\Cf}{\mathbf{C}}

\newcommand{\C}{\mathbb{C}}
\newcommand{\gqs}{\geqslant}
\newcommand{\gr}{\operatorname{Gr}}
\newcommand{\id}{\operatorname{id}}
\newcommand{\Img}{\operatorname{Im}}
\newcommand{\ind}{\operatorname{Ind}}
\newcommand{\lqs}{\leqslant}
\newcommand{\la}{\langle}
\newcommand{\N}{\mathbb{N}}
\newcommand{\ra}{\rangle}

\newcommand{\R}{\mathbb{R}}
\newcommand{\supp}{\operatorname{supp}}
\newcommand{\tmu}{t^{-1}}
\newcommand{\vep}{\varepsilon}
\newcommand{\what}[1]{\widehat{#1}}
\newcommand{\wt}[1]{\widetilde{#1}}

\newtheorem{theorem}{Theorem}[section]
\newtheorem{lemma}[theorem]{Lemma}

\theoremstyle{definition}
\newtheorem{definition}[theorem]{Definition}
\newtheorem{example}[theorem]{Example}

\theoremstyle{remark}

\numberwithin{equation}{section}

\begin{document}

\title{An Imprimitivity Theorem for Partial Actions}

\author{Dami\'an Ferraro}

\address{Departamento de Matem\'atica y Estad\'istica del Litoral, Universidad
de la Rep\'ublica, Gral. Rivera 1350. Salto. Uruguay.}

\thanks{This work was supported by the CSIC and started when the author was
a member of the Centro de Matem\'atica of the Facultad de
Ciencias, Universidad de la Rep\'ublica, Uruguay.}

\email{damianferr@gmail.com}


\date{\today}

\begin{abstract}
We define proper, free and commuting partial actions on upper semicontinuous
bundles of $C^*-$algebras. With such, we construct the $C^*-$algebra induced by
a partial action and a partial actions on that algebra. Using those action we
give a generalization, to partial actions, of Raeburn's Symmetric Imprimitivity
Theorem \cite{Raeburn}.
\end{abstract}

\maketitle

\section*{Introduction}
The main idea of this article appear in the following example. Let $\be$ be a
continuous, free and proper action of a locally compact and Hausdorff
(LCH) group $G$ on a LCH space $Y$. This gives us a continuous action of $G$
on the continuous functions vanishing at infinity of $Y,$ $C_0(Y)$. If $Y/G$ is
the orbit space of $Y$, then Green's Theorem \cite{Rieffel-TeoGreen}
implies $C_0(Y/G)$ is strongly Morita equivalent to the crossed product
$C_0(Y)\rtimes_{\be} G.$

Now consider an open subset $X\subset Y$ such that $\cup\{\be_t(X):\ t\in
G\}=Y$. Lets call $\al$ the restriction of $\be$ to $X$. That is, for every
$t\in G$ set $\al_t:X\cap \be_{\tmu}(X)\to X\cap \be_t(X),\ x\mapsto \be_t(x).$
This is an example of a partial action. Now consider the open set
$\Gamma:=\{(t,x)\in G\times X\ |\ \be_{\tmu}(x)\in X\}\subset G\times Y.$ The
crossed product $C_0(X)\rtimes_\al G$ is the closure of $C_c(\Gamma)\subset
C_c(G,Y)$ in $C_0(Y)\rtimes_\be G.$ It is strongly Morita equivalent to
$C_0(Y)\rtimes_\be G$ \cite{Fernando,Abadie-Marti}.

Putting all together, we conclude that $C_0(X)\rtimes_\al G$ is strongly Morita
equivalent to $C_0(Y/G).$ The objective of the present work is to generalize the
previous idea to the case where we just know $X,$ $G$ and $\al.$ That is,
$\al$ is a partial action of $G$ on $X.$

The outline of this work is as follows. In Section 1 we give the
definitions of free, proper and commuting partial actions and prove some basic
results involving those concepts, it is based on \cite{Fernandophd,Fernando}. In
the second section we define partial actions on upper semicontinuous
$C^*-$bundles and, with such, construct the induced $C^*-$algebra of a partial
action and partial actions on those induced algebras. Here we follow Raeburn's
work \cite{Raeburn}. Finally, we prove our main theorem which is a
generalization, to partial actions, of Raeburn's Theorem \cite{Raeburn}. On a
first read, to understand the basic ideas, we suggest the reader to consider
bundles of the form $X\times \C$ ($X$ is a topological space and $\C$ the
complex numbers) with trivial action on $\C.$

\section{Properties of Partial Actions}\label{Properties of Partial Actions}

Through this work the letters $G$, $H$ and $K$ will denote LCH topological
groups and $X,\ Y$ topological spaces. When any additional topological property
is required it will be explicitly mentioned (this will never happen for the
groups).

This section is a brief resume of some results contained in \cite{Fernando} and
in the PHD Thesis \cite{Fernandophd}, for that reason some proof will be
omitted. We start by recalling the definition of partial action.

\begin{definition}[\cite{Exel,Fernando,Fernandophd}]\label{definicion ap}
A pair $\al=(\{X_t\}_{t\in H},\{\al_t\}_{t\in H})$ is a partial action of  $H$
on $X$ if, for every $t,s\in H$:
\begin{enumerate}
\item $X_t$ is a subset of $X$ and $X_e=X$ ($e$ being the identity of $H$).
\item $\al_t:X_t\to X_{t^{-1}}$ is a bijection and $\al_e=\id_X$ (the identity
on $X$).
\item If $x\in X_{t^{-1}}$ and $\al_t(x)\in X_{s^{-1}}$, then $x\in
X_{(st)^{-1}}$ and $\al_{st}(x)=\al_s\circ \al_t(x).$
\end{enumerate}
\end{definition}

The domain of $\al$ is the set $\Gamma_{\al}:=\{ (t,x)\in H\times X\ |\ x\in
X_{t^{-1}} \}.$ Recall $\al$ is continuous if $\Gamma_{\al}$ is open in $H\times
X$ and the function, also called $\al$, $\Gamma_{\al}\to X$, $(t,x)\mapsto
\al_t(x)$, is continuous. The graph of the partial action $\al$, $\gr(\al)$, is
the graph of the function $\al:\Gamma_{\al}\to X$. We say $\al$ has closed graph
if $\gr(\al)$ is closed in $H\times X\times X$.

Take two continuous partial actions of $H,$ $\al$ and $\be,$ on the spaces $X$
and $Y$ respectively. A morphism $f:\al\to\be$ is a continuous function $f:X\to
Y$ such that for every $t\in H:$ $f(X_t)\subset Y_t$ and the restriction of
$\be_t\circ f$ to $X_{\tmu}$ equals $f\circ \al_t.$

Given $\be$ as before and a non empty open set $Z\subset Y,$ the restriction of
$\be$ to $Z$ is the continuous partial action of $H$ on $Z$ given by
$\gamma_t:Z\cap \be_{\tmu}(Z)\to Z\cap \be_t(Z),$ $z\mapsto \be_t(z).$

Up to isomorphism of partial actions, every continuous partial action can be
obtained as a restriction of a global action. That is, given
$\al$ as before there exits a global and continuous action of $H$ on a
topological space $Y,$ $\be,$ and an open set $Z\subset Y$ such that $\al$ is
isomorphic to the restriction of $\be$ to $Z.$ If in addition $Y=\cup\{\be_t(Z)\
|\ t\in H \},$ we say $\be$ is an \textit{enveloping action} of $\al.$
Enveloping actions exists and are unique up to isomorphism of (partial)
actions \cite{Fernando,Fernandophd}. The enveloping\footnote{We say ``the''
enveloping action because it is unique up to isomorphisms.} action of $\al$ is
denoted $\al^e$ and the space where it acts $X^e,$ we also think $X$ is an open
set of $X^e$ and $\al$ is the restriction of $\al^e$ to $X.$

The orbit of a subset $U\subset X$ by $\al$ is the set $HU:=\cup\{\al_t(U\cap 
X_{\tmu})\ |\ t\in G\}.$ The orbit of a point $x\in X$ is the orbit of the set
$\{x\}$ and is denoted $H x.$ If we want to emphasize the name of the action we
write $\al H x.$ The orbits of two points are equal or disjoint and the union of
all of them is equal to $X.$ With this partition of $X$ we construct the
quotient space $X/H$ with the quotient topology, this is the \textit{orbit
space} of $\al.$ The canonical projection $X\to X/H$ is continuous, surjective
and open. The function $X/H\to X^e/H, $ $\al H x\to \al^e H x$ is a
homeomorphism.

Raeburn's Symmetric Imprimitivity Theorem involves free, proper and commuting
actions. We now give the corresponding definitions for partial actions. We
refer the reader to \cite{Fernandophd} to a more detailed exposition of these
concepts.

The \textit{stabilizer} of a point $x\in X$ is the set $H_x:=\{t\in H\ | \ x\in
X_{\tmu},\ \al_t(x)=x\}.$ It is easy to see that $H_x$ is a subgroup of $H,$
not necessarily closed if the action is not global. A partial action is
\textit{free} is the stabilizer of every point is the set $\{e\}.$ A partial
action is free if and only if it's enveloping action is free.

The next concept we define is commutativity. We will have two continuous partial
actions, $\al$ and $\be$, of $H$ and $K$, on $X$. As we do not want any
confusions, we will use the notation $\al_s:X^H_{s^{-1}}\to X^H_s$ and
$\be_t:X^K_{t^{-1}}\to X^K_t$, for $s\in H$ and $t\in K$.

We say $\al$ and $\be$ \textit{commute} if for every $(s,t)\in H\times K$ (i)
$\al_s(X^H_{s^{-1}}\cap X^K_t)=\be_t(X^K_{t^{-1}}\cap X^H_s)$ and
(ii) $\al_s\circ \be_t(x)=\be_t\circ \al_s(x)$, for every $x\in \al_{s^
{-1}}(X^H_s\cap X^K_{t^{-1}}).$ This definition expresses the fact that we can
compute $\al_s\be_t(x)$ if and only if we can compute $\be_t\al_s(x)$, and in
that case $\al_s\be_t(x)=\be_t\al_s(x)$. As we can see, if both actions are
global, this is the usual notion of commuting actions.

Recall \cite{Exel-Laca-Quigg} a subset $U\subset X$ is said $\al-$invariant if
$\al_t(X^H_{t^{-1}}\cap U)\subset U$, for every $t\in H$. Condition (i) of the
previous definition implies $X^K_s$ is $\al-$invariant for every $s\in K$.

An important property of commuting global actions is that we can define an
action of the product group, this is also true for partial actions.

\begin{lemma}[cf. \cite{Fernandophd} Proposi\c c\~ao 4.35]\label{accion
producto}
If $\al$ and $\be$ commute then there is a continuous partial action, $\al\times
\be$, of $G:=H\times K$ on $X$ such that, for every $(s,t)\in G,$
\begin{enumerate}
\item $X_{(s,t)}=\be_t(X^K_{t^{-1}}\cap X^H_s)=\al_s(X^H_{s^{-1}}\cap X^K_t)$
and
\item $\al\times \be_{(s,t)}=\al_s\circ\be_t$.
\end{enumerate}
\end{lemma}
\begin{proof}
The fact that $\mu:=\al\times \be$ is a partial action (not necessarily
continuous) is an easy consequence of the fact that $\al$ and $\be$ are
commuting partial actions (\cite{Fernandophd} Proposi\c c\~ao 4.35). We just
have to deal with the continuity.

To show $\Gamma_{\mu}$ is open in $G\times X$ notice that the set
$\Gamma^{-1}_{\be}:=\{ (t,x)\in K\times X\ | \ x\in X^K_t \}$ is open in
$K\times X$, and define:
\begin{gather}
\pi_H:H\times K\times X\to H\times X,\ (h,k,x)\mapsto (h,x), \notag \\
\pi_K:H\times K\times X\to K\times X,\ (h,k,x)\mapsto (k,x),\textrm{ and} \notag
\\
F: \pi_K^{-1}(\Gamma_{\be})\to \pi_K^{-1}(\Gamma^{-1}_{\be}), \ (h,k,x)\mapsto
(h,k,\be_k(x)). \notag
\end{gather}

It is easy to see that the three functions are continuous. So, the domain and
range of $F$ are open in $H\times K\times X$ and
$\Gamma_{\mu}=F^{-1}(\pi_K^{-1}(\Gamma^{-1}_{\be})\cap \pi_H^{-1}(\Gamma_{\al})
)$ is open in $G\times X$. Finally, the continuity of $\mu:\Gamma_{\mu}\to X$
follows from that of $\al:\Gamma_{\al}\to X$ and $\be:\Gamma_{\be}\to X$.
\end{proof}

Here is another property of partial actions we will use.

\begin{lemma}\label{accion en espacio de orbitas}
If $\al$ and $\be$ commute then there is a partial action of $H$ on $X/K$,
called $\what{\al}$, such that for every $s\in H$
\begin{enumerate}
\item $(X/K)_s:=KX^H_s$ and
\item $\what{\al}_s(Kx)=K\al_s(x)$ for any $x\in X^H_{s^{-1}}$.
\end{enumerate}
\end{lemma}
\begin{proof}
The first step is to show we can define $\what{\al}$ as in (1) and (2). Define
the function $F:\Gamma_{\al}\to H\times X/K$ as $F(s,x)=(s,Kx)$. This map is
open and continuous. The domain of $\what{\al}$ will be the image of $F$,
$\Gamma_{\what{\al}}:=\Img(F),$ which is an open set.

Now define $S:\Gamma_{\al}\to X/K$ in such a way that $(s,x)\mapsto K\al_s(x)$,
and consider (on $\Gamma_{\al}$) the equivalence relation $u\sim v$ if
$F(u)=F(v)$. The function $S$ is constant in the classes of $\sim$, and the
quotient space $\Gamma_{\al}/\sim$ is homeomorphic to $\Gamma_{\what{\al}}$
trough the map defined by $F$. So, there is a unique continuous map
$\Gamma_{\what{\al}}\to X/K$ such that $(s,Kx)\mapsto K\al_s(x)$. This is the
partial action $\what{\al}$ we are looking for.

It remains to be shown that $\what{\al}$ is a partial action. Properties (1) and
(2) of Definition \ref{definicion ap} are easy to prove, for (3) recall every
$X^H_s$ is $\be-$invariant.
\end{proof}

Assume for a moment we have a continuous global action of $H$ on $X$. It is
immediate that it's domain, being equal to $H\times X$, is a closed and open
(clopen) set of $H\times X$. That is not always the case for partial actions.

\begin{definition}
A partial action, $\al$, of $H$ on $X$ has \textit{closed domain} if
$\Gamma_{\al}$ is closed in $H\times X$.
\end{definition}

\begin{lemma}\label{dominio cerrado equivalencias}
If $X$ is Hausdorff and $\al$ is continuous, the following conditions are
equivalent:
 \begin{enumerate}
 \item $\al$ has closed domain.
 \item The enveloping space $X^e$ is Hausdorff and $X$ is closed in $X^e$.
 \end{enumerate}
\end{lemma}
\begin{proof}
We start by proving (1)$\Rightarrow$(2). Recall \cite{Fernando} $X^e$ is
Hausdorff if $\al$ has closed graph. Consider the function $F: H\times
X\times X\to H\times X,\ (s,x,y)\mapsto (s,x).$ The set $F^{-1}(\Gamma_\al)$ is
closed in $H\times X\times X$. Now, $\gr(\al)$ is closed in $F^{-1}(\Gamma_\al)$
because it is the pre-image of the diagonal $\{(x,x)\ |\ x\in X\}\subset X\times
X$ by the continuous function $F^{-1}(\Gamma_\al)\to X\times X,$ $(t,x,y)\mapsto
(\al_t(x),y).$ This implies $\al$ has closed graph.

To show $X$ is closed in $X^e$ take a net contained in $X$, $\{x_i\}_{i\in I},$
converging to a point $x\in X^e$. There exists $t\in H$ such that $\al^e_t(x)\in
X$. By the continuity of $\al^e$ there is an $i_0$ such that
$(t^{-1},\al^e_t(x_i))\in \Gamma_{\al}$ for $i\gqs i_0$. Then
$(t^{-1},\al^e_t(x)),$ being the limit of $\{(t^{-1},\al^e_t(x_i)) \}_{i\gqs
i_0}$, belongs to $\in\Gamma_{\al}$. Finally $x=\al_{t^{-1}}\al^e_t(x)\in X.$

For the converse notice three facts: the topology of $H\times X$ is the topology
relative to $H\times X^e$, $(\al^e)^{-1}(X)$ is closed in $H\times X^e$ and
$\Gamma_{\al}=H\times X\cap (\al^e)^{-1}(X)$. So we clearly have that
$\Gamma_{\al}$ is closed in $H\times X$.
\end{proof}

The previous lemma characterizes the continuous partial actions with closed
domain on Hausdorff spaces, as those arising as the restriction of a global
action on a Hausdorff space to a clopen set.

\begin{lemma}\label{dominio cerrado accion producto}
Given two continuous and commuting partial actions, both with closed domain, the
partial action of the product group (as defined on Lemma \ref{accion producto})
has closed domain.
\end{lemma}
\begin{proof}
In the proof of Lemma \ref{accion producto} we showed $\Gamma_{\al\times\be}$ is
open, use the same arguments changing the word ``open" for ``closed".
\end{proof}

A dynamical system (DS for short) is a tern $(Y,H,\be)$ where $\be$ is a
continuous action of $H$ on $Y$, where $H$ and $Y$ are LCH. The natural
extension to partial actions is the following one.

\begin{definition}
The tern $(X,H,\al)$ is a \textit{partial dynamical system} (PDS) if $\al$ is a
continuous partial action of $H$ on $X$ and both ($H$ and $X$) are LCH.
\end{definition}

\begin{lemma}\label{envolvente es SD sii grafico cerrado}
If $(X,H,\al)$ is a PDS then $(X^e,H,\al^e)$ is a DS if and only if $\al$ has
closed graph.
\end{lemma}
\begin{proof}
By Theorem 1.1. of \cite{Fernando} every point of $X^e$ has a neighbourhood
homeomorphic to $X$. So, every point of $X^e$ has a local basis of compact
neighbourhoods. As $\al$ is continuous and $H$ is LCH, $(X^e,H,\al^e)$ is a DS
if and only if $X^e$ is Hausdorff. By Proposition 1.2. of \cite{Fernando} $X^e$
is Hausdorff if and only if $\al$ has closed graph.
\end{proof}

A PDS $(X,H,\al)$ is \textit{proper} if the function $F_{\al}:\Gamma_{\al}\to
X$, $(t,x)\mapsto (x,\al_t(x))$, is proper (the pre-image of a compact set is
compact). This definition, and part of the next Lemma, are taken from
\cite{Fernandophd}.

\begin{lemma}\label{equivalencia ap propia}
Given a PDS $(X,H,\al),$ the following statements are equivalent:
\begin{enumerate}
\item The system is proper.
\item Every net contained in $\Gamma_{\al}$, $\{(t_i,x_i)\}_{i\in I}$, such that
$\{(x_i,\al_{t_i}(x_i))\}_{i\in I}$ converges to some point of $X\times X$, has
a subnet converging to a point of $\Gamma_{\al}$.
\item $\al$ has closed graph and the enveloping DS $(X^e,H,\al^e)$ is proper.
\end{enumerate}
\end{lemma}
\begin{proof}
The equivalence between (1) and (3) is proved in \cite{Fernandophd}
(Proposic\~ao 4.62). The equivalence between (1) and (2) is proved as in
Lemma 3.42 of \cite{Williams}.
\end{proof}

It is a known fact that the orbit space of a proper DS is a LCH space, this is
also true for PDS.

\begin{lemma}
If $(X,H,\al)$ is a proper PDS then $X/H$ is a LCH space.
\end{lemma}
\begin{proof}
By the previous Lemma $(X^e,H,\al^e)$ is a proper DS. So, $X^e/H$ is LCH. But
$X/H$ is homeomorphic to $X^e/H$ and so is LCH.
\end{proof}

The next result follows immediately from the previous ones.

\begin{lemma}\label{orbitas para PDS}
Let $(X,H,\al)$ and $(X,K,\be)$ be commuting PDS (that is, $\al$ and $\be$
commute). If $\be$ is proper then $(X/K,H,\what{\al})$ is a PDS, where
$\what{\al}$ is the partial action defined on Lemma \ref{accion en espacio de
orbitas}.
\end{lemma}

\section{Partial actions on bundles of $C^*-$Algebras.}

The definition of upper semicontinuous $C^*-$bundle we are going to use is
Definition C.16 of \cite{Williams} (notice that we do not require the base space
to be Hausdorff). From now on $\Bf=\{B_x\}_{x\in X}$ and $\Cf=\{C_y\}_{y\in Y}$
will be upper semicontinuous $C^*-$bundles. The projections of $\Bf$ and $\Cf$
will be denoted $p:B\to X$ and $q:C\to Y$, respectively.

The set of continuous and bounded sections of the bundle $\Bf$ will be denoted
$C_b(\Bf)$ (this notation differs from that of \cite{Williams}). Similarly,
$C_0(\Bf)$ is the set of continuous sections vanishing at infinity (C.21
\cite{Williams}) and $C_c(\Bf)$ the set of continuous sections of compact
support. When $X$ is a LCH space, $C_b(\Bf)$ and $C_0(\Bf)$ are $C^*-$algebras
with the supremum norm and $C_c(\Bf)$ is a dense $*-$sub algebra of $C_0(\Bf)$.

\begin{definition}\label{definicion ap en fibrados}
A partial action of $H$ on $\Bf$ is a pair $(\al, \cdot )$, where $\al$ and
$\cdot$ are continuous partial actions of $H$ on $B$ and $X$, respectively,
satisfying
\begin{enumerate}
\item $p^{-1}(X_t)={}_tB$ for every $t\in H$. Here ${}_tB$ is the range of
$\al_t$.
\item $p$ is a morphism of partial actions (Definition 1.1 of \cite{Fernando}).
\item The restriction of $\al_t$ to a fiber is a morphism of $C^*-$algebras, for
each $t\in H$.
\end{enumerate}

Notice that $\cdot$ is determined by $\al$. For that reason, with abuse of
notation, we name $\al$ the pair $(\al,\cdot)$. We say $\al$ is global if the
partial action on the total space is a global action or, what is the same, if
$\cdot$ is global.

The domain of the partial action on the total and base space will be denoted
$\Gamma(B,\al)$ and $\Gamma(X,\al)$, respectively.
\end{definition}

\begin{example}\label{acion en fibrado trivial}
Let $(X,H,\cdot)$ be a PDS and $(A,H,\gamma)$ a $C^*-$DS, that is, $A$ is a
$C^*-$algebra and $\gamma:H\to Aut(A)$ is a strongly continuous action. With
such define the trivial bundle $p:A\times X\to X$, where $p(a,x)=x$. All the
fibers of this bundle, called $\Bf$, are isomorphic to $A$ by the maps $A\to
B_x$, $a\mapsto (a,x)$. We define a global action of $H$ on $\Bf$ by setting
$\al_t:A\times X_{t^{-1}}\to A\times X_t $, $(a,x)\mapsto (\gamma_t(a),t\cdot
x).$
\end{example}

I would like to emphasize that, from now on, we are going to use the letters
$\al$ and $\be$ for actions on total spaces. The actions on the base spaces will
be denoted $\cdot$ and $\star$. We will write $\al_t(a)$ and $t\cdot x$,
similarly with $\be$ and $\star$.

If $\be=(\be,\star)$ is a partial action of $H$ on $\Cf$, a morphism
$(F,f):\al\to\be$ is a pair of continuous functions, $F:B\to C$ and $f:X \to Y$,
such that: both are morphism of partial actions, $q\circ F=f\circ p$ and the
restriction of $F$ to each fiber is a morphism of $C^*-$algebras. Naturally, the
composition of morphisms is the composition of functions (on each coordinate).

Following \cite{Fernando} we can define the restriction of actions. Let
$\be=(\be,\star)$ be a global action of $H$ on $\Bf$ and $U$ an open subset of
$X$. Consider the restriction bundle $\Bf_U=\{B_u\}_{u\in U}$ with the partial
action $\be_U$, which is the pair formed by the restriction of the actions of
$H$ to $p^{-1}(U)$ and $U$. Notice we have obtained a partial action because
$p^{-1}(U)\cap \be_t(p^{-1}(U))=p^{-1}(U\cap t\star U )$, for every $t\in H$.

Rephrasing Theorem 1.1 of \cite{Fernando} we get

\begin{theorem}\label{construction of enveloping bundle}
For every continuous partial action $\al$ of $H$ on an upper semicontinuous
$C^*-$bundle $\Bf$, there exists a tern $(\iota_X,\iota_B,\al^e)$ such that
$\al^e$ is an action of $H$ on an upper semicontinuous $C^*-$bundle $\Bf^e$, and
$(\iota_X,\iota_B):\al\to \al^e$ is a morphism, such that for any morphism
$\psi:\al\to \be$, where $\be$ is an action of $H$ (on an upper semicontinuous
$C^*-$bundle), there exists a unique morphism $\psi^e:\al^e\to \be$ such that
$\psi^e\circ(\iota_X,\iota_B)=\psi$.

Moreover, the pair $(\iota_X,\iota_B,\al^e)$ is unique up to canonical
isomorphisms, and
\begin{enumerate}
\item $\iota_X(X)$ is open in $X^e$.
\item $(\iota_X,\iota_B): \al \to (\al^e)_{\iota_X(X)}$ is an isomorphism.
\item $X^e$ is the orbit of $\iota_X(X)$.
\item $\Bf^e$ is a continuous $C^*-$bundle if and only if $\Bf$ is.
\end{enumerate}
\end{theorem}
\begin{proof}
Let $(\iota_X,\cdot^e)$ and $(\iota_B,\al^e)$ be the pairs given by Theorem 1.1
of \cite{Fernando} for $\cdot$ and $\al$. We also have a morphism $p^e:\al^e\to
\cdot^e$. Notice $p^e$ is surjective because is a morphism and the orbit of
$\iota_X(X)$ equals $X^e$. Again, as $B^e$ is the orbit of $\iota_B(B)$ and
$p^e$ is a morphism, to prove $p^e$ is open we only have to see that $p^e\circ
\al^e_t\circ \iota_B$ is open, for every $t\in H$. But this is true because, if
$U$ is open in $B$ \begin{equation*}  p^e\circ \al^e_t\circ \iota_B(U) =
t\cdot^e(\iota_X(p(U))), \end{equation*} the last being an open set.

We have proved $B^e$ fibers over $X^e$. We now give a structure of $C^*-$algebra
to each fiber of $B^e$. Let $x$ be an element of $X^e$, take $t\in H$ such that
$t\cdot^e x\in \iota_X(X)$ and define the $C^*-$structure on $B^e_x$ as the
unique making $\al^e_{t^{-1}}\circ\iota_B:B_{\iota_X^{-1}(t\cdot^e x)}\to B^e_x$
an isomorphism of $C^*-$algebras. This is independent of the choice of $t$
because $\al$ acts as isomorphism of $C^*-$algebras on the fibers of $\Bf.$

To prove the norm of $\Bf^e$ is semicontinuous notice that, given $\vep>0,$ the
set $\{b\in B^e:\ \|b\|<\vep \}$ equals the open set $\bigcup_{t\in H}
\al^e_t\circ\iota_B(\{b\in B:\ \|b\|<\vep\}).$ In fact, a similar argument shows
the norm of $\Bf^e$ is continuous if and only if the norm of $\Bf$ is
continuous. This suffices to prove property (4) of the thesis.

We now indicate how to prove the continuity of the product, for the other
operations there are analogous proofs. Set $D^e:=\{(a,b)\in B^e\times B^e:\
p^e(a)=p^e(b)\}$. We prove the continuity of $D^e\to B^e,$ $(a,b)\mapsto ab,$
locally. Fix $(a,b)\in D^e$, we may assume $p(a)=t\cdot^e x$ for some $x\in X$
and $t\in H$. The product is continuous on $(a,b)$ because
$U:=(\al_t\circ\iota_B(B)\times \al_t\circ\iota_B(B))\cap D^e$ is open in $D^e$,
and the restriction of the product to $U$ is the continuous function
\begin{equation*} (c,d)\mapsto \al^e_t\circ\iota_B\left[ 
(\iota_B)^{-1}\circ\al^e_{t^{-1}}(c)+(\iota_B)^{-1}\circ\al^e_{t^{-1}}(d)  
\right]. \end{equation*}

Up to here we have constructed an upper semicontinuous $C^*-$bundle
$\Bf^e=\{B^e_x\}_{x\in X^e}$. By the previous construction we also have that
$(\al^e,\iota^e)$ is a global action of $H$ on $\Bf^e$, and
$(\iota_X,\iota_B):\al\to\al^e$ is a morphism. Except for property (2) of the
thesis, everything follows immediately from the previous constructions and
Theorem 1.1 of \cite{Fernando}.

To prove property (2) it suffices to see that
$(p^e)^{-1}(\iota_X(X))=\iota_B(B)$. We clearly have the inclusion $\supset$,
for the other one let $b\in (p^e)^{-1}(\iota_X(X))$. We may suppose
$b=\al^e_t(\iota_B(c))$ for some $c\in B$ and $t\in H$. As
$p^e(b)=p^e(\al^e_t(\iota_B(c)))\in \iota_B(B)$, we have \begin{equation*}
p^e(b)=p^e(\al^e_t(\iota_B(c)))=t\cdot^ e\iota_X(p(c))\in \iota_X(X).
\end{equation*} So, $p(c)\in X_{t^{-1}}$. This implies $b=\iota_B(\al_t(c))\in
\iota_B(B)$.
\end{proof}

The non commutative analogue of PDS's are the $C^*-$PDS's, they are terns
$(A,G,\gamma)$ formed by a $C^*-$algebra $A$, a LCH group $G$ and a partial
action $\gamma$ of $G$ on $A$ (Definition 2.2 of \cite{Fernando}, for a more
general definition see \cite{Exel}).

We know every PDS gives us a $C^*-$PDS with commutative algebra
\cite{Fernando}. Following that construction, we are going to use partial
actions on upper semicontinuous $C^*-$bundles over LCH spaces to construct
partial actions on the $C^*-$algebras $C_0(\Bf)$. The ideals are of the form $
C_0(\Bf,U):=\{f\in C_0(\Bf)\ |\ f(x)=0_x\textrm{ if }x\notin U \} $, for open
sets $U\subset X$.

\begin{theorem}\label{cxPDS dado por p.a. en fibrado}
Let $X$ be a LCH space, $\Bf=\{B_x\}_{x\in X}$ an upper semicontinuous
$C^*-$bundle and $\al$ a continuous partial action of $H$ on $\Bf$. Then
$(C_0(\Bf),H,\wt{\al})$ is a $C^*-$PDS, where

\begin{enumerate}
\item $C_0(\Bf)_t=C_0(\Bf,X_t)$, for every $t\in H$.
\item If $f\in C_0(\Bf)_{t^ {-1}}$ then $\wt{\al}_t(f)(x)=\al_t(f(t^{-1}\cdot
x))$ if $x\in X_t$ and $0_x$ otherwise.
\end{enumerate}
\end{theorem}

\begin{proof}
First of all we have to show that, given $t\in H$ and $f\in C_0(\Bf)_{t^ {-1}}$,
the function $\wt{\al}_t(f)$ belongs to $C_0(\Bf)_t$. It is clear that
$\wt{\al}_t(f)$ is a section that vanishes outside $X_t$. Besides, the function
$X_t\to\R,$ $x\mapsto \|\what{\al}_t(f)(x)\|,$ being equal to $X_t\to\R,$
$x\mapsto \|f(t^{-1}\cdot x)\|,$ vanishes at infinity.

Clearly $\what{\al}_t(f)$ is continuous on $X_t$ and in the interior of the
complement of $X_t$. To prove the continuity of $\what{\al}_t(f)$ it suffices to
show that given a net $\{x_i\}_{i\in I}\subset X_t$ converging to a point
$x\notin X_t$, we have $\|\wt{\al}_t(f)(x_i)\|\to 0$. Notice that the function
$X_{t^{-1}}\to\R,$ $y\mapsto \|f(y)\|,$ vanishes at infinity and the net
$\{t^{-1}\cdot x_i\}_{i\in I}$ is eventually outside every compact of
$X_{t^{-1}}$, we conclude $\|\wt{\al}_t(f)(x_i)\|=\|f(t^{-1}\cdot x_i)\|\to 0.$

The next step is to show $\what{\al}$ is a partial action (Definition
\ref{definicion ap}). We omit the proof of this fact because it is an easy task.

To prove $\{C_0(\Bf)_t\}_{t\in H}$ is a continuous family \cite{Exel}, let $U$
be an open set of $C_0(\Bf)$ and fix $t\in H$ such that $C_0(\Bf)_t\cap U\neq
\emptyset$. By the Urysohn Lemma we can find $g\in C_0(\Bf)_t\cap U$ with
compact support. As the domain of the partial action on $X$ is an open set,
there is an open set containing $t$, $V$, such that $X_r$ contains the support
of $g$ for every $r\in V$. Then $V$ is an open set containing $t$ and contained
in $\{r\in H\ |\ C_0(\Bf)\cap U\neq\emptyset \}.$

Now we deal with the continuity of $\wt{\al}$. Let $\{(t_i,f_i)\}_{i\in I}$ be a
net contained in $\Gamma_{\wt{\al}}$ converging to $(t,f)\in \Gamma_{\wt{\al}}$.
Given $\vep > 0$ there exists $g\in C_c(\Bf)$, with support contained in
$X_{t^{-1}}$, such that $\|f-g\|<\frac{\vep}{3}$ (by the Urysohn Lemma).

We can find an $i_0\in I$ such that $\supp (g)\subset X_{t_i^{-1}}$ and
$\|f_i-g\|<\frac{\vep}{3}$, for every $i\gqs i_0$. Then, for every $i\gqs i_0$

\begin{align}
\| \wt{\al}_t(f)-\wt{\al}_{t_i}(f_i) \| & \lqs \| \wt{\al}_{t_i}(f_i-g) \|+\|
\wt{\al}_{t_i}(g)-\wt{\al}_t(g)\|+ \| \wt{\al}_t(f-g) \| \notag \\
         &<\frac{2\vep}{3}+ \| \wt{\al}_{t_i}(g)-\wt{\al}_t(g)\|.\notag
\end{align}

To complete the proof it suffices to see that $ \lim_i \|
\wt{\al}_{t_i}(g)-\wt{\al}_t(g)\|=0$. To this purpose let $D$ be a compact
containing $t\cdot \supp(g)$ on it's interior, and contained in $X_t$. We may
find $i_1$ (larger than $i_0$) such that $t_i\cdot \supp (g)\subset D$ and
$D\subset X_{t_i}$, for every $i\gqs i_1$. Given $i\gqs i_1$ we have
\begin{equation*} \| \wt{\al}_{t_i}(g)-\wt{\al}_t(g)\|=\sup \{ \|
\al_{t_i}(g(t_i^{-1}\cdot x))-\al_t(g(t^{-1}\cdot x))\|:\ x\in
D\}.\end{equation*}

As $D$ is compact, it suffices to prove $\wt{\al}_{t_i}(g)$ converges point wise
to $\wt{\al}_t(g)$, which is an easy consequence of Lemma C.18 of
\cite{Williams} and the continuity of $\al$ and $g$.
\end{proof}

The definitions of proper, free and commuting partial actions on bundles are the
following ones.

\begin{definition}
Given an upper semicontinuous $C^*-$bundle over a LCH space and a partial action
of a LCH group on the bundle, we say the partial action is \textit{free},
\textit{proper}, \textit{has closed graph} or \textit{has closed domain} if the
partial action on the base space has the respective property. Similarly, given
two partial actions on an upper semicontinuous $C^*-$bundle we say they commute
if the partial actions on the total and base space commute.
\end{definition}

Relating the concepts of enveloping action, in the contexts of $C^*-$algebras
and bundles, we have the following result.

\begin{theorem}\label{envolvente de ap C_0(fibrados)}
Let $\al$ be a partial action of $H$ on the upper semicontinuous $C^*$-bundle
$\Bf=\{B_x\}_{x\in X}$. If $X$ is LCH, $\al$ has closed graph, $\al^e$ is the
enveloping action of $\al$ and $\Bf^e$ the enveloping bundle, then
$(C_0(\Bf^e),H,\wt{\al^e})$ is the enveloping system of $(C_0(\Bf),H,\wt{\al})$
(Definition 2.3 of \cite{Fernando}). So $C_0(\Bf)\rtimes_{\wt{\al}}H$ is a
hereditary and full sub $C^*-$algebra of $C_0(\Bf^e)\rtimes_{\wt{\al^e}} H$. In
particular those crossed products are strongly Morita equivalent.
\end{theorem}

\begin{proof}
By Theorem \ref{construction of enveloping bundle} we may suppose $X\subset
X^e$, $\Bf\subset \Bf^e$, $p^e(\Bf)=X$ and that $p$ is the restriction of $p^e$
to $\Bf$. Now, by Lemma \ref{envolvente es SD sii grafico cerrado}, $X^e$ is
LCH. This considerations allows us to identify the bundle $\Bf$ with the
restriction of $\Bf^e$ to $X$, which gives $C_0(\Bf)=C_0(\Bf^e,X)$. We have
identified $C_0(\Bf)$ with an ideal of $C_0(\Bf^e).$

We also have, for every $t\in H$,
\begin{align}
\wt{\al^e}_t(C_0(\Bf^e,X))\cap C_0(\Bf^e,X) & = C_0(\Bf^e,t\cdot X)\cap
C_0(\Bf^e,X)=C_0(\Bf^e,X\cap t\cdot X)\notag \\
           &=C_0(\Bf^e,X_t)=C_0(\Bf)_t;\notag
\end{align}
and clearly the restriction of $\wt{\al^e}_t$ to $C_0(\Bf)_{t^{-1}}$ equals
$\wt{\al}_t$. So, using Corollary 1.3 of \cite{Abadie-Marti}, the only thing
that remains to be showed is that the space generated by the $\wt{\al^e}-$orbit
of $C_0(\Bf)$ is dense in $C_0(\Bf^e)$.

We show every continuous section with compact support of $\Bf^e$ is a finite sum
of points in the orbit of $C_0(\Bf)$. Fix $f\in C_c(\Bf^e)$. The support of $f$
has an open cover by sets of the form $t\cdot^e X$, $t$ varying in $H$. We can
find $t_1,\ldots,t_n\in H$ and $h_1,\ldots,h_n\in C_c(X^e)$ such that: $0\lqs
h_1+\cdots +h_n\lqs 1$, the support of $h_i$ is contained in $t_i\cdot X$
($i=1,\ldots,n$) and $h_1(x)+\cdots +h_n(x)=1$ if $x\in \supp (f)$. Defining
$f_i(x)=h_i(x)f(x)$ ($i=1,\ldots,n$), we have $f=f_1+\cdots +f_n$ and
$g_i:=\wt{\al^e}_{t_i^{-1}}(f_i)\in C_c(\Bf)$ for every $i=1,\ldots,n$. Besides,
$f=\wt{\al^e}_{t_1}(g_i)+\cdots + \wt{\al^e}_{t_n}(g_n)$, that gives the desired
result.
\end{proof}

We can reproduce most of the results of Section \ref{Properties of Partial
Actions} in this context. For example, the next Theorem is a direct consequence
of Lemma \ref{accion producto}.

\begin{theorem}\label{accion producto en fibrado}
Given an upper semicontinuous $C^*-$bundle $\Bf$ and commutative partial
actions, $(\al,\cdot)$ and $(\be,\star)$ of $H$ and $K$ on $\Bf$, respectively,
the pair $(\al,\cdot)\times (\be,\star):=(\al\times \be,\cdot\times \star)$ is a
partial action of $H\times K$ on $\Bf$.
\end{theorem}

Writing $\al=(\al,\cdot)$ and $\be=(\be,\star)$, the product $\al\times \be$ is
the one defined in the previous Theorem.

\subsection{Orbit bundle}\label{fibrado cociente} There is a notion of ``orbit
bundle", analogous to the notion of ``orbit space", but to construct it we have
to consider proper and free partial actions.

Fix an upper semicontinuous $C^*-$bundle over a LCH space, $\Bf=\{B_x\}_{x\in
X}$, and a proper and free partial action, $\al$, of $H$ on $\Bf$. Let $B/H$ and
$X/H$ be the orbit spaces and $\pi_{B}:B\to B/H$ and $\pi_X:X\to X/H$ be the
orbit maps. As $p$ is a morphism of partial actions, there is a unique
continuous (also open and surjective) function $p_{\al}:B/H\to X/H$ such that
$\pi_X\circ p=p_{\al}\circ\pi_B$.

We want to equip $B/H$ with operations making $\Bf/H:=(B/H,X/H,p_\al)$ an upper
semicontinuous $C^*$-bundle. To do this first notice that, given $Hx\in X/H$,
the fiber $(B/H)_{xH}$ is homeomorphic to $B_x$ trough the restriction of
$\pi_B$ to $B_x$ (because the partial action on $X$ is free). Call that map
$h_x:B_x\to (B/H)_{xH}$. Define the structure of $C^*-$algebra of $(B/H)_{xH}$
in such a way that $h_x$ is an isomorphism of $C^*-$algebras. This definition is
independent of the choice of $x$ because, if $Hy=Hx$, then $h_y^{-1}\circ
h_x:B_x\to B_y$ is an isomorphism. As it is the restriction of $\al_t$ to $B_x$,
$t\in H$ being the unique such that $x\in X_{t^{-1}}$ and $t\cdot x=y$.

The norm of $B/H$ is the function $\|\ \|:B/H\to\R,$ $Ha\mapsto \|a\|$. To prove
it is upper semicontinuous let $\vep$ be a positive number. The set $\{ Hb\in
B/H\ |\ \|Hb\|<\vep \}$ is open because it equals the open set $\pi_B(\{b\in B\
|\ \|b\|<\vep \})$. Similarly, we prove $\|\ \|:B/H\to\R$ is continuous if $\|\
\|:B\to\R$ is continuous.

To prove the continuity of the product and the sum name $D$ the set of points
$(a,b)\in B\times B$ such that $Ha=Hb$. If $(a,b)\in D$ there is a unique $t\in
H$, which we name $t(p(a),p(b))$, such that $p(a)\in X_{t^{-1}}$ and $t\cdot
p(a)=p(b)$. Hence, $\al_t(a)$ and $b$ are in the same fiber, we define
$S(a,b):=\al_t(a)+b$ and $P(a,b):=\al_t(a)b$.

To prove the continuity of $S$ and $P$ we only have to prove the continuity of
the function
\begin{equation*}F:\{(x,y)\in X\times X:\ Hx=Hy\}\to \Gamma(X,\al), \
F(x,y)=(t(x,y),x).\end{equation*} Call $D_X$ the domain of $F$.

Consider the function $R:\Gamma(X,\al)\to X\times X$ given by $R(t,x)=(x,t\cdot
x)$, this is a continuous, proper and injective function between LCH spaces.
Such functions are homeomorphisms over its image, but the image of $R$ is $D_X$
and $F=R^{-1}$. So, $F$ is continuous.

Once we have proved the continuity of $S$ and $P$, using the freeness of the
partial action on $X$, we prove they are constant in the classes of the
equivalence relation on $D:$ $(a,b)\sim (c,d)$ if $Ha=Hc$ and $Hb=Hd$. The space
$D/\sim$ is (homeomorphic to) $D':=\{(a,b)\in B/H\times B/H: \
p_{\al}(a)=p_{\al}(b)\}$, and the functions defined by $S$ and $P$ on $D'$ are
exactly the sum and product of $B/H$. We have proved they are continuous, for
the rest of the operations we proceed in a similar way.

The last step, to show $\Bf/H$ is an upper semicontinuous $C^*-$bundle is to
prove it satisfies the following property: for every net $\{b_i\}_{i\in
I}\subset B/H$ such that $\|b_i\|\to 0$ and $p_{\al}(b_i)\to z$, for some $z\in
X/H$, we have $b_i\to 0_z$.

Let $\{b_i\}_{i\in I}$ be a net as before, it suffices to show it has a subnet
converging to $0_z$. There is a net in $B$, $\{a_i\}_{i\in I}$, such that
$b_i=Ha_i$ for every $i\in I$. We have that $Hp(a_i)=p_{\al}(b_i)\to Hx$, where
$x\in X$ is such that $Hx=z$. As the orbit map $X\to X/H$ is open and
surjective, Proposition 13.2 Chapter II of \cite{Fell-Doran} implies there is a
subnet $\{a_{i_j}\}_{j\in J}$ and a net $\{t_j\}_{j\in J}\subset H$ such that
$p(a_{i_j})\in X_{t_j^{-1}}$ and $t_j\cdot p(a_{i_j})\to x$. This implies
$a_{i_j}\in {}_{t_j^{-1}}B$ and $p(\al_{t_{i_j}}(a_{i_j}))\to x$. But also $\|
\al_{t_{i_j}}(a_{i_j}) \|=\|b_{i_j}\|\to 0$, so, $\al_{t_{i_j}}(a_{i_j})\to
0_x$. Finally, as $b_{i_j}= H\al_{t_{i_j}}(a_{i_j})$, $\pi_B(0_x)=H0_x=0_z$ and
$\pi_B$ is continuous, $0_z$ is a limit point of $\{b_{i_j}\}_{j\in J}$.

\begin{definition}
The \textit{orbit bundle} of $\Bf$ by $\al$ is the upper semicontinuous
$C^*-$bundle $\Bf/H$ constructed before.
\end{definition}

\begin{example}
Consider the situation of Example \ref{acion en fibrado trivial} where the
action of $H$ on $A$ is the trivial one ($\gamma_t=\id_A$ for every $t\in H$)
and
the system $(X,H,\cdot)$ is free and proper. Then the quotient bundle $\Bf/H$ is
isomorphic to the trivial bundle $A\times X/H.$ Notice that $C_0(\Bf/H)$ is
isomorphic to $C_0(X/H,A)$, the set of continuous functions from $X/H$ to $A$
vanishing at infinity.
\end{example}

Our next goal is to identify $C_0(\Bf/H)$ with a $C^*-$sub algebra of
$C_b(\Bf)$. Every function $f\in C_b(\Bf),$ which is also a morphism of partial
actions, induces a continuous and bounded section $\ind_b(f):X/H\to B/H$, given
by $Hx\mapsto Hf(x)$.

The induced algebra $\ind_b(\Bf,\al)$ is the subset of $C_b(\Bf)$ formed by all
the sections which are also morphism of partial actions. There is a natural map
\begin{equation*} \ind_b:\ind_b(\Bf,\al)\to C_b(\Bf/H),\ f\mapsto
\ind_b(f).\end{equation*}

Similarly, the algebra $\ind_0(\Bf,\al)$ is the pre image of $C_0(\Bf/H)$ under
$\ind_b$. The function $\ind_0$ is simply the restriction of $\ind_b$ to
$\ind_0(\Bf,\al)$.

In fact, the induced algebras are $C^*-$sub algebras of $C_b(\Bf)$. To prove
this it suffices to show $\ind_b(\Bf,\al)$ is a $C^*-$sub algebra and to notice
$\ind_b$ is a morphism of $C^*-$algebras.

The non trivial fact is that $\ind_b(\Bf,\al)$ is closed in $C_b(\Bf)$. Assume
$\{f_n\}_{n\in \N}$ is a sequence contained in $\ind_b(\Bf,\al)$ converging to
$f$. Choose some $t\in H$ and $x\in X_{t^{-1}}$. Even if $B$ is not Hausdorff,
$B_x$ and $B_{t\cdot x}$ are, so we have the following equalities
\begin{equation*} \al_t(f(x))=\lim_n \al_t(f_n(x))=\lim_n f_n(t\cdot x)=f(t\cdot
x).\end{equation*}

\begin{theorem}\label{identificacion algebras inducidas}
The functions \begin{equation*}\ind_b:\ind_b(\Bf,\al)\to C_b(\Bf/H)\ and \
\ind_0:\ind_0(\Bf,\al)\to C_0(\Bf/H)\end{equation*} are isomorphism of
$C^*-$algebras.
\end{theorem}
\begin{proof}
The only thing to prove is that $\ind_b$ is surjective (it is injective because
is an isometry). Fix $g\in C_b(\Bf)$, we will construct $f\in \ind_b(\Bf,\al)$
such that $\ind_b(f)=g$.

As the action on the base space is free, for every $x\in X$ there is a unique
$f(x)\in B_x$ such that $Hf(x)=g(Hx)$. Clearly $f$ is a bounded section.

To prove $f$ is continuous, let $\{x_i\}_{i\in I}$ be a net contained in $X$
converging to $x\in X$. It suffices to find a subnet $\{x_{i_j}\}_{j\in J}$ such
that $f(x_{i_j})\to f(x).$ By the continuity of $g$ the net $\{Hf(x_i)\}_{i\in
I}$ has $Hf(x)$ as a limit point. As the orbit map $B\to B/H$ is open, there is
a subnet $\{x_{i_j}\}_{j\in J}$ and a net $\{t_j\}_{j\in J}$ such that $\{
(t_j,f(x_{i_j})) \}_{j\in J}\subset \Gamma(B,\al)$ and $\al_{t_j}(f(x_{i_j}))\to
f(x)$. This implies $\{ (t_j,x_{i_j}) \}_{j\in J}\subset \Gamma(X,\al)$ and
$t_j\cdot x_{i_j}\to x$. Then $t_j=t(x_{i_j},t_j\cdot x_{i_j})\to t(x,x)=e$ (see
the construction of the orbit bundle in Section \ref{fibrado cociente}).
Finally, the net $\{ f(x_{i_j})\}_{j\in J}$, being equal to $\{
\al_{t_j^{-1}}\al_{t_j}(f(x_{i_j})) \}_{j\in J}$, has $f(x)$ as a limit point.

It remains to prove $f$ is a morphism of partial actions. Clearly $f(X_t)\subset
{}_tB$ for every $t\in H$. Now take $t\in H$ and $x\in X_{t^{-1}}$. The points
$f(t\cdot x)$ and $\al_t(f(x))$ are, both, the unique point of $B_{t\cdot x}$ in
the class of $g(Hx)$, so they are equal.
\end{proof}

\begin{theorem}\label{AP en algebra inducida}
Let $\Bf=\{B_x\}_{x\in X}$ be an upper semicontinuous $C^*-$bundle over a LCH
space and $\al$ and $\be$ be partial actions of $H$ and $K$ on $\Bf$,
respectively. If $\al$ is free and proper, then $(\ind_0(\Bf,\al),K,\what{\be})$
is a $C^*-$PDS where
\begin{enumerate}
\item $\ind_0(\Bf,\al)_t:=\{ f\in \ind_0(\Bf,\al): \ x\mapsto \|f(x)\|
\textrm{vanishes outside } X^H_t\}.$
\item For every $f\in \ind_0(\Bf,\al)_{t^{-1}}$
$\what{\be}_t(f)(x)=\be_t(f(t^{-1}\star x))$ if $x\in X^K_t$ and $0_x$
otherwise.
\end{enumerate}
\end{theorem}
\begin{proof}
Let $\Bf/H$ be the orbit bundle. As $\al$ commutes with $\be$, using Lemmas
\ref{accion en espacio de orbitas} and \ref{orbitas para PDS}, we define a
partial action, $\mu$, of $K$ on $\Bf/H$.

By Theorem \ref{cxPDS dado por p.a. en fibrado}, $\mu$ defines a $C^*-$PDS
$(C_0(\Bf/H),K,\wt{\mu})$. Lemma \ref{identificacion algebras inducidas} ensures
the map $\ind_0:\ind_0(\Bf,\al)\to C_0(\Bf/H)$ is an isomorphism. Notice
$\ind_0(\Bf,\al)_t$ is the pre image of $C_0(\Bf/H)_t$. The partial action of
the thesis is the unique making $\ind_0:\what{\be}\to \wt{\mu}$ an isomorphism
of partial actions.
\end{proof}

\section{Morita equivalence}
In our last section we prove our main theorem, which is a generalization of
Raeburn's and Green's Symmetric Imprimitivity Theorems
\cite{Raeburn,Rieffel-TeoGreen}. The first task is to translate Raeburn's result
to the language of actions on bundles.

Consider two $C^*-$DS $(A,H,\gamma)$ and $(A,K,\delta)$, and two proper and free
DS $(X,H,\cdot)$ and $(X,K,\star)$. Assume also that the actions on $A$ and $X$
commute. On the trivial bundle $\Bf=A\times X$ define the actions of $H$ and $K$
as in Example \ref{acion en fibrado trivial}, call them $\al$ and $\be$,
respectively.

Let $\ind \gamma$ be the induced $C^*$-algebra defined as in \cite{Raeburn}. We
have an isomorphism $\rho:\ind \gamma\to  \ind_0(\Bf,\al)$, given by
$\rho(f)(x)=(f(x),x)$. This isomorphism takes the action of $K$ on $\ind \gamma$
(as defined on \cite{Raeburn}) into the action $\what{\be}.$ By using Raeburn's
Theorem we conclude that $\ind_0(\Bf,\al)\rtimes_{\what{\be}}K $ is strongly
Morita equivalent to $\ind_0(\Bf,\be)\rtimes_{\what{\al}}H$. Our purpose is to
give a version of this result for partial actions. We will write $A\sim_M B$
whenever $A$ and $B$ are strongly Morita equivalent $C^*-$algebras
\cite{Rieffel-MorEq}.

\subsection{The main Theorem}
From now on we work with two LCH topological groups, $H$ and $K$, an upper
semicontinuous $C^*-$bundle with LCH base space, $\Bf=\{B_x\}_{x\in X}$, and two
continuous, free, proper and commuting partial actions, $\al$ and $\be$, of $H$
and $K$ on $\Bf$, respectively.

We want to give conditions under which we can say that
$\ind_0(\Bf,\al)\rtimes_{\what{\be}}K$ is strongly Morita equivalent to
$\ind_0(\Bf,\be)\rtimes_{\what{\al}}H$. For global actions, with some additional
hypotheses on the group and the base space, this is proved in \cite{Hu-Rae-Wil}, \cite{Kasparov} or \cite{KlMhQgWl}. In fact, the proof of the next Theorem is a minor
modification of Raeburn's proof the Symmetric Imprimitivity Theorem
\cite{Raeburn}.

\begin{theorem}\label{pre main}
If $\al$ and $\be$ are global actions then
\begin{equation*}\ind_0(\Bf,\al)\rtimes_{\what{\be}}K\sim_M
\ind_0(\Bf,\be)\rtimes_{\what{\al}}H.\end{equation*}
\end{theorem}
\begin{proof}
Define $E:=C_c(H,\ind_0(\Bf,\be))$ and $F:=C_c(K,\ind_0(\Bf,\al))$, viewed as
dense $*-$sub-algebras of the respective crossed products. Define also
$Z:=C_c(\Bf),$ which will be an $E-F-$bimodule with inner products; whose
completion implements the equivalence between
$\ind_0(\Bf,\al)\rtimes_{\what{\be}}K$ and
$\ind_0(\Bf,\be)\rtimes_{\what{\al}}H.$

For $f,g\in Z$, $b\in E$ and $c\in F$ define
\begin{gather}
b\cdot f(x):= \int_H b(s)(x)\wt{\al}_s(f)(x)\Delta_H(s)^{1/2}
ds,\label{operaciones modulo 1}\\
f\cdot c(x):= \int_K \wt{\be}_t (f c(t^{-1}))(x) \Delta_K(t)^{-1/2}
dt,\label{operaciones modulo 2}\\
{}_{E}\la f,g\ra(s)(x):=\Delta_H(s)^{-1/2} \int_K \wt{\be}_t\left(
f\wt{\al}_s(g^*)\right)(x)dt,\label{operaciones modulo 3}\\
\la f,g\ra_{F}(t)(x):=\Delta_K(t)^{-1/2} \int_H \wt{\al}_s\left(
f^*\wt{\be}_t(g)\right)(x)ds.\label{operaciones modulo 4}
\end{gather}

The integration is with respect to left invariant Haar measures; $\Delta_H$ and
$\Delta_K$ are the modular functions of the groups. Here $\wt{\al}$ and
$\wt{\be}$ are the partial actions defined on Theorem \ref{cxPDS dado por p.a.
en fibrado}. 

We now justify the fact that $b\cdot f\in Z$. The function $H\to C_0(\Bf)$,
given by $s\mapsto b(s)\wt{\al}_s(f)\Delta_H(s)^{1/2}$, is continuous (Theorem
\ref{cxPDS dado por p.a. en fibrado}). Besides, it's support is contained in the
support of $b$ and so we can integrate it. This integral is exactly $b\cdot f$.
Finally, notice $\supp(b\cdot f)\subset \{s\cdot x: \ (s,x)\in \supp(b)\times
\supp(f)\},$ the last being a compact set.

To prove (\ref{operaciones modulo 3}) defines an element of $E$ we proceed as
follows. Fixed $s\in H$ and $x\in X$ the function $K\to B_x$, given by $t\mapsto
\wt{\be}_t\left(f\wt{\al}_s(g^*)\right)(x)$, is continuous with support
contained in the compact $\{ t\in K: \ t^{-1}\star x\in \supp(f)\}.$ So, the
function is integrable. The value of that integral is ${}_{E}\la f,g\ra(s)(x)$.

We now prove ${}_{E}\la f,g\ra$ is continuous, what we do locally. Fix some
$s_0\in H$ and $x_0\in X$. Take compact neighbourhoods, $V$ of $s_0$ and $W$ of
$x_0$. The bundle $\Bf_W$ will be the restriction of $\Bf$ to $W$. Define the
function $F:V\times H\to C(\Bf_W)$ by $F(s,t)(x)=
\wt{\be}_t\left(f\wt{\al}_s(g^*)\right)(x)$. As the action of $K$ on $X$ is
proper, $F$ has compact support. By integrating, with respect to the second
coordinate, we get a continuous function $R\in C(V,C(\Bf_W))$, defined by
$R(s)=\int_K F(s,t)d\mu_K(t)$ (\cite{Fell-Doran} II.15.19).

Fixing $(s,x)\in V\times W$, we have $R(s)(x)={}_{E}\la f,g\ra(s)(x)$. From this
follows the continuity of ${}_{E}\la f,g\ra$.

An easy calculation shows ${}_{E}\la f,g\ra(s)(t\cdot x)=\be_t({}_{E}\la
f,g\ra(s)(x))$, for every $t\in K$ and $x\in X$. Besides, if ${}_{E}\la
f,g\ra(s)(x)\neq 0_x$, then $x$ belongs to the $K-$orbit of $\supp(f)$, and $s$
to the compact set $\{ s\in H:\ s\cdot\supp(g)\cap\supp(f)\neq\emptyset \}$. We
have proved that ${}_{E}\la f,g\ra\in E.$

The computations needed to prove equations (\ref{operaciones modulo
1})-(\ref{operaciones modulo 4}) define an equivalence bi-module are the same as
in \cite{Raeburn} or \cite{Williams}. For the construction of the approximate
unit, analogous to that of Lemma 1.2 of \cite{Raeburn}, follow the proof of
Proposition 4.5 of \cite{Williams}, recalling $\ind_H^P(A,\be)$ plays the role
of our $\ind_0(\Bf,\be).$
\end{proof}

Our next step is to let $\al$ and $\be$ to be partial, but to have the same
result we need additional hypotheses, which are trivially satisfied in the
previous case.

Let $\al\times \be$ be the partial action given by Theorem \ref{accion producto
en fibrado}. Now, by Theorem \ref{construction of enveloping bundle}, we have an
enveloping action $(\al\times \be)^e$ and an enveloping bundle $\Bf^e$. We can
assume $\Bf$ is the restriction of $\Bf^e$ to $X\subset X^e$.

For the action given by $(\al\times \be)^e$ on $X^e$ we will use the notation
$(s,t)x$, for $(s,t)\in H\times K$ and $x\in X^e.$

Define $\sigma$ and $\tau$ as the restriction of $(\al\times \be)^e$ to $H$ and
$K$, respectively (identify $H$ with $H\times \{e\}\subset H\times K$). It is
immediate that $(\al\times \be)^e=\sigma\times \tau$, $\sigma$ and $\tau$
commute, and that $\al$ ($\be$) is the restriction of $\sigma$ ($\tau$) to
$\Bf$.

The next is the main Theorem of this article.

\begin{theorem}\label{main}
If $\al\times\be$ has closed graph and $\sigma$ and $\tau$ are proper then
\begin{equation*}\ind_0(\Bf,\al)\rtimes_{\what{\be}}K\sim_M
\ind_0(\Bf,\be)\rtimes_{\what{\al}}H.\end{equation*}
\end{theorem}
\begin{proof}
To show that $\sigma$ (and also $\tau$) is free. Assume $(s,e)x=x$ for some
$s\in H$ and $x\in X^e$. As $X^e$ is the $H\times K-$orbit of $X$, there exists
$(h,k)\in H\times K$ such that $(h,k)x\in X$. Notice that
$(hsh^{-1},e)(h,k)x=(h,k)x\in X\cap (hsh^{-1},e)^{-1}X$, so, $hsh^{-1}\cdot
(h,k)x=(h,k)x$ and $hsh^{-1}=e$. We conclude $s=e$.

As $\al\times \be$ has closed graph, $\Bf^e$ is an upper semicontinuous
$C^*-$bundle over a LCH space. The hypotheses, together with Theorems \ref{AP en
algebra inducida} and \ref{pre main}, imply
$\ind_0(\Bf^e,\sigma)\rtimes_{\what{\tau}}K$ is strongly Morita equivalent to
$\ind_0(\Bf^e,\tau)\rtimes_{\what{\sigma}} H.$ The proof of our Theorem will be
completed if we can show that $\ind_0(\Bf,\al)\rtimes_{\what{\be}}K$ is strongly
Morita equivalent to $\ind_0(\Bf^e,\sigma)\rtimes_{\what{\tau}}K,$ because, by
symmetry, the same will hold changing $\al$ for $\be,$ $\sigma$ for
$\tau$ and $H$ for $K.$

Tracking back the construction of $\what{\be}$ and $\what{\tau}$, to Theorem
\ref{AP en algebra inducida}, we notice that
$\ind_0(\Bf,\al)\rtimes_{\what{\be}}K$ is isomorphic to
$C_0(\Bf/H)\rtimes_{\wt{\mu}}K$ and $\ind_0(\Bf^e,\sigma)\rtimes_{\what{\tau}}K$
isomorphic to $C_0(\Bf^e/H)\rtimes_{\wt{\nu}}K.$ Here $\mu$ and $\nu$ are the
partial actions of $K$ on $\Bf/H$ and $\Bf^e/H$ given by Lemma \ref{accion en
espacio de orbitas}, respectively. Meanwhile, $\wt{\mu}$ and $\wt{\nu}$ are the
one given by Theorem \ref{cxPDS dado por p.a. en fibrado}. Putting all together,
by Theorem \ref{envolvente de ap C_0(fibrados)}, it suffices to prove $\nu$ is
the enveloping action of $\mu$.

Consider the map $B\to B^e/H$, given by $b\mapsto Hb$. This is an open and
continuous map, it is also constant in the $\al-$orbits. So it defines a unique
map $F:B/H\to B^e/H$, given by $Hb\mapsto Hb$ (this is not the identity map). It
turns out this function is continuous, open, injective and maps fibers into
fibers. In an analogous way we define $f:X/H\to X^e/H$, which has the same
topological properties.

Recalling the construction of $\mu$ and $\nu$, it is easy to show $(F,f):\mu\to
\nu$ is a morphism. To show that $\mu^e=\nu$ it suffices to prove only two
things. Namely, that $f((X/H)_t)=f(X/H)\cap tf(X/H)$ for every $t\in K$ (we
adopted the notation $tz$ for the action of $t\in K$ on $z\in X^e/H$) and that
the $K-$orbit of $f(X/H)$ is $X^e/H$.

For the first one notice that
\begin{equation*} f((X/H)_t)=HX^K_t=HX\cap H(e,t)X=HX\cap tHX=f(X/H)\cap
tf(X/H). \end{equation*}

The second equality of the previous formula is not immediate, but the inclusion
$\subset$ is. For the other one assume $y\in HX\cap H(e,t)X$. Then there exists
$x,z\in X$ such that $y=Hx=H(e,t)z$. There is some $s\in H$ such that
$x=(s,e)(e,t)z=(s,t)z$. So $x\in X\cap (s,t)X=s\cdot (X^H_s\cap X^K_t)$,
$(s^{-1},e)x\in X^K_t$ and $y=K(s^{-1},e)x\in HX^K_t$.

To show $X^e/H$ is the $K-$orbit of $f(X/H),$ notice that
\begin{equation*} \bigcup_{t\in K}tf(X/H)= \bigcup_{t\in K}tHX=H \bigcup_{s \in
H}\bigcup_{t\in K}(s,t)X =HX^e=X^e/H.\end{equation*}

We have proved $\nu$ is the enveloping action of $\mu$, by Theorem
\ref{envolvente de ap C_0(fibrados)} $C_0(\Bf/H)\rtimes_{\wt{\mu}}K$ is strongly
Morita equivalent to $C_0(\Bf^e/H)\rtimes_{\wt{\nu}}K$. This completes the proof
of our main theorem.
\end{proof}

The next Theorem is a consequence of the previous one, it has the advantage of
not making any mention to $\sigma$ nor $\tau$.

\begin{theorem}\label{con dominio cerrado}
If $\al$ and $\be$ have closed domain then \begin{equation*}
\ind_0(\Bf,\al)\rtimes_{\what{\be}}K\sim_M
\ind_0(\Bf,\be)\rtimes_{\what{\al}}H.\end{equation*}
\end{theorem}
\begin{proof}
We check the hypotheses of the previous theorem are satisfied. To show
$\al\times\be$ has closed graph notice it has closed domain (Lemma \ref{dominio
cerrado accion producto}) and use Lemma \ref{dominio cerrado equivalencias}.
Finally, we only have to show $\sigma$ and $\tau$ are proper. It is enough to
show $\sigma$ is proper, for that purpose we use Lemma \ref{equivalencia ap
propia}.

Let $\{(s_i,x_i)\}$ be a net in $H\times X^e$ such that $\{(x_i,(s_i,e)x_i)\}$
converges to the point $(x,y)\in X^e\times X^e.$ It is enough to show
$\{s_i\}_{i\in I}$ has a converging subnet. We may assume $(s,t)x\in X$ and
$(h,k)y\in X$, for some $(h,k),(s,t)\in H\times K$.

There is an $i_0$ such that, for $i\gqs i_0$, $(s,t)x_i$ and $(h,k)(s_i,e)x_i$
belong to $X.$ For $i\gqs i_0$ define $u_i=(s,t)x_i$. By the construction of
$\al\times \be$ and because $(hs_is^{-1},kt^{-1})u_i\in X$, we have that $u_i$
is an element of the clopen set $tk^{-1}\star (X^K_{k^{-1}t}\cap
X^H_{ss_ih^{-1}})$. Defining $v_i:=(e,k^{-1}t)u_i$ for every $i\gqs i_0$, we
have that $v_i\in X$. So, the limit $\lim_i v_i$ is an element of $X$ (recall
$X$ is clopen in $X^e$).

The net $\{ (hs_is^{-1},v_i) \}_i$ is contained in $\Gamma(X,\al)$ and
$\{(v_i,hs_is^{-1}\cdot v_i )\}_i$ has a limit point. Then $\{hs_is^{-1}\}_i$
has a converging subnet, an so $\{s_i\}_i$ has a converging subnet. We conclude
$\sigma$ is proper, and we are done.
\end{proof}

\bibliographystyle{hsiam}
\bibliography{FerraroBibliocorto}

\end{document}